\documentclass[runningheads,a4paper]{llncs}\usepackage{latexsym}
\usepackage{graphicx}
\usepackage{amsmath}
\usepackage{amssymb,bbm,color}
\usepackage{color,soul}
\usepackage{enumitem}
\usepackage{csquotes}
\usepackage{tikz}
\usetikzlibrary{calc,decorations.pathreplacing}
\usepackage{ctable}

\usepackage[pdftex,colorlinks=true,linkcolor=blue,unicode,pdfpagelabels]{hyperref}

\let\RE\Re
\let\Re=\undefined
\DeclareMathOperator{\Re}{\RE e}
\let\IM\Im
\let\Im=\undefined
\DeclareMathOperator{\Im}{\IM m}

\newcommand{\R}{\mathbbm R}

\newcommand{\N}{\mathbbm N}
\newcommand{\Z}{\mathbbm Z}

\newcommand{\ve}{\varepsilon}
\newcommand{\argmin}{\text{arg min}}

\newcommand{\ud}{u^\delta}

\newcommand{\abs}[1]{\left| #1 \right|}
\newcommand{\set}[1]{\left\{ #1 \right\}}
\newcommand{\norm}[1]{\left\| #1 \right\|}

\newtheorem{thm}{Theorem}

\newtheorem{rem}[thm]{Remark}

\urldef{\mailsa}\path|{guozhi.dong,otmar.scherzer}@univie.ac.at|

\begin{document}
\title{Nonlinear Flows for Displacement Correction and Applications in Tomography}

\titlerunning{Nonlinear Flows for Displacement Correction}
\author{Guozhi Dong\inst{1}  \and Otmar Scherzer \inst{1,2}
}
\authorrunning{G. Dong, and O. Scherzer}

\institute{Computational Science Center,\\ University of Vienna,\\
Oskar-Morgenstern-Platz 1, 1090 Wien, Austria\\
\mailsa\\
\and
Johann Radon Institute for Computational and\\Applied Mathematics (RICAM),\\ Austrian Academy of Sciences,\\
Altenbergerstrasse 69, A-4040 Linz, Austria
}


\maketitle

\begin{abstract}
In this paper we derive nonlinear evolution equations associated 
with a class of non-convex energy functionals which can be used for 
correcting displacement errors in imaging data.
We study properties of these filtering flows 
and provide experiments for correcting angular perturbations 
in tomographical data.
\end{abstract}

\section{Introduction}
In this paper, we are investigating variational methods and partial differential equations for filtering displacement errors 
in imaging data. Such types of errors appear when measurement data are \emph{sampled} erroneously. 
In this work we consider measurement data $\ud$, which are considered perturbations of 
an ideal function $u$ at random sampling locations $(x_1+d_i(x_i),x_2)$: That is, 
\begin{equation}
\label{eq:dis_function}
\ud(\mathbf{x})=u(x_1+d_i(x_i),x_2), \text{ for } \mathbf{x}=(x_1,x_2)\in  \R^2.
\end{equation}
A particular case of a displacement error $x_1 \to d_1(x_1)$ appears in Computerized Tomography 
(here the $x_1$-component denotes the $X$-ray beam direction ($\theta$ below)) when the angular sampling 
is considered erroneous. In this case the recorded data are
\begin{equation}
\label{eq:angle_dis}
y^\delta (\theta,l)=R[f](\theta+d_1(\theta),l).
\end{equation}
Here $R[f]$ denotes the \emph{Radon transform} or \emph{X-ray transform} of the function $f$, 
and $\theta$ and $l$ denote the beam direction and beam distance, respectively.

Displacement errors of the form $d_2(x_2)$ are jittering errors, and the filtering and compensation 
of such has been considered in \cite{She04,KanShe06,Nik09p,Nik09,DonPatSchOek15}.

Our work is motivated by \cite{LenSch11,LenSch09}, where partial differential equations 
for denoising image data destructed by \emph{general} sampling errors of the form
\[\ud(\mathbf{x})=u(\mathbf{x}+\mathbf{d}) 
\text{ with } \mathbf{d}:\R^2 \rightarrow \R^2,\]
have been stated. The nonlinear evolution equations have been derived by 
mimicking a convex semi-group for a non-convex energy. The PDEs from \cite{LenSch11,LenSch09} 
revealed properties similar to the mean curvature motion equation \cite{EvaSpr91,EvaSpr92}. 
In comparison to \cite{LenSch11,LenSch09}, here, we are considering displacement errors in the 
$x_1$-component only.

The paper is organized as follows:
In Section \ref{sec:regularization}, we review the state of the art of non-convex regularization models for 
sampling error corrections: 
In particular, we comment on algorithms for recovering different types of displacements in a discrete setting.
In Section \ref{sec:flows}, we consider nonlinear filtering flows motivated from non-convex regularization energies.
For these flows we present numerical experiments, which suggest particular properties of the PDEs.
Finally, we present an application of the novel filtering techniques for correcting tomographical image data 
with errors in the beam directions.

\section{Non-convex regularization models}
\label{sec:regularization}
We study the following problem: Let $i=1$ or $2$ be fixed.
Given noisy image data $\ud$, the goal is to simultaneously recover the ideal function $u$ \text{ and } the 
displacement $d_i$ satisfying \eqref{eq:dis_function}. 
Figure \ref{fig:displacement} shows the effect of displacement errors $d_1$ and $d_2$ on some image data, respectively. 
The two displacement errors result in image data with orthogonal visual perturbations.
\begin{figure}[bht]
\centering
\includegraphics[width=0.45\textwidth]{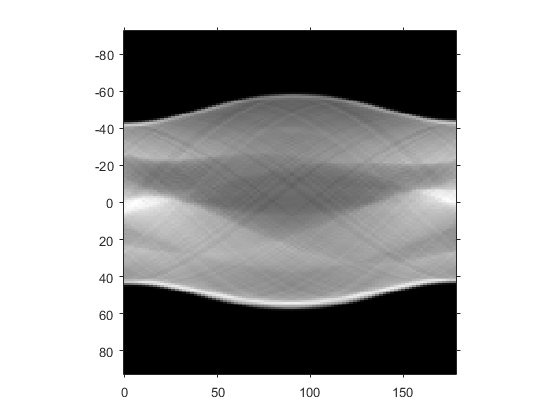}  \includegraphics[width=0.45\textwidth]{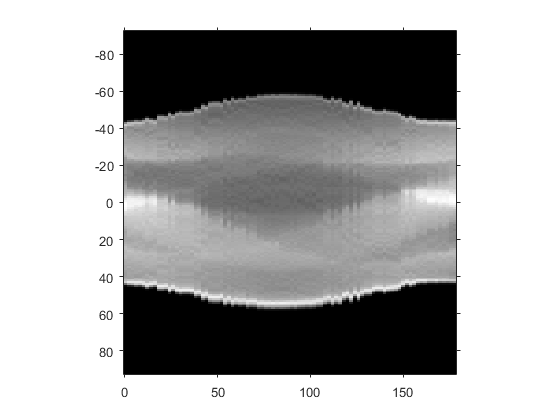}  \\
\includegraphics[width=0.45\textwidth]{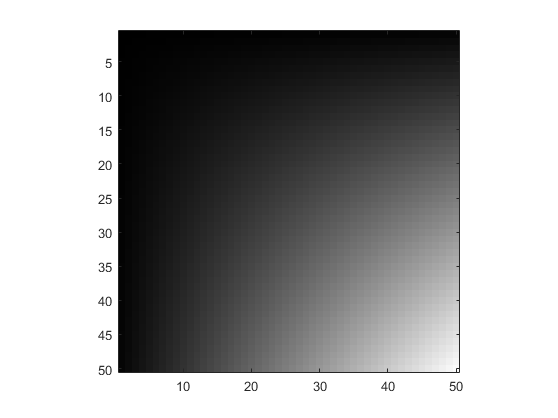}  \includegraphics[width=0.45\textwidth]{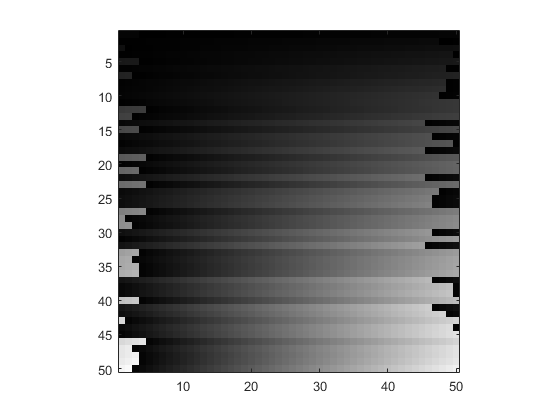} 
\caption{\sl{Top:} Non-disturbed and angular disturbed sinogram of the Shepp-Logan phantom (in the notation of the Radon transform 
         $x_1=\theta$ and $x_2=l$). \sl{Bottom:} Original and jittered image. The two different kinds of errors 
         results appear as visually complementary data perturbations. This observation is the basis of regularization methods 
         \eqref{eq:non_convex}, 
         where jittering is corrected by regularizing in $x_2$ direction and tomographic problems are corrected by filtering 
         in $x_1$ direction.}
\label{fig:displacement}
\end{figure}

To this end we consider the following optimization problem:
\begin{equation}
\label{eq:problem1}
\begin{array}{lll}
& \min_{d_i}& \mathcal{J}_i(d_i|\ud):=\norm{\partial_i^k \ud(x_1-d_i(x_i),x_2)},\\
& \text{ such that }& \abs{d_i}\leq M.
\end{array}
\end{equation}
Here $\norm{\cdot}$ denotes some proper quasi-norm or norm of functions $v : \R^2 \to \R$ and the choice of 
$k \in \N$ depends on the a-priori information on regularity of the function $u$.

Below, we are considering \emph{discrete} and \emph{continuous} optimization formulations of Problem \eqref{eq:problem1}.
\begin{itemize}
 \item In the discrete optimization problem, $d_i(x)  \in \Z$ denotes the pixel displacements of columns 
($i=1$) and rows ($i=2$), respectively. The image data $u$ can be represented as a matrix with non-negative 
integer values, that is $u \in \N_0^{l \times n}$, where each matrix entry denotes the discrete image intensity 
of the image data at a position $(c,r) \in \set{1,\ldots,l} \times \set{1,\ldots,n}$.
Moreover, the derivatives are considered in the sense of finite differences.
\item In a continuous formulation, $d_i \in \R$ and the image is considered a function $u: \R^2 \to \R$.
\end{itemize}

For $i=2$, the discrete optimization problem \eqref{eq:problem1} has been investigated in \cite{Nik09p,Nik09} 
and the continuous optimization problem has been studied in \cite{She04,DonPatSchOek15}.

\subsection{Discrete optimization algorithms}
The computational complexity of the discrete optimization algorithms varies significantly for solving the discrete Problem \eqref{eq:problem1} 
with $i=1$ and $i=2$.
\begin{itemize}
\item Nikolova \cite{Nik09p,Nik09} introduced a highly efficient optimization algorithm with exhaustive search for the case $i=2$.
      The complexity is $\mathcal{O}(M n)$ for $u \in \N_0^{l \times n}$. 
\item For $i=1$ the discrete optimization problem \eqref{eq:problem1} is an \emph{assignment} problem. Even the 
      simplified problem of alignment of the columns in each of the $\lfloor \frac{l}{M} \rfloor$ non-intersecting 
      sub-blocks has already an exponential complexity $\mathcal{O}((M!)^{\lfloor \frac{l}{M} \rfloor}) $ by exhaustive search.
 
      We note that the complexity of assignment problems depends on the properties of the given cost functionals.
      For linear assignment problem (see e.g. \cite{BurDelMar09}), the Hungarian algorithm \cite{Mun57} has a complexity 
      $\mathcal{O}(l^3)$. However, nonlinear assignment problems \cite{Vos00}, such as Problem \eqref{eq:problem1}, are 
      usually \emph{NP-hard}.
\end{itemize} 

\subsection{Continuous models}
The discrete optimization algorithm for solving Problem \eqref{eq:problem1} can be used to correct for large displacement errors.
Small (including subpixel) displacement errors can be corrected for by using partial differential equations: If the displacement $d_i$ 
is small, following \cite{LenSch11,DonPatSchOek15}, we consider a first order approximation of the continuous data 
$u: \R^2 \to \R$. Then
\begin{equation}
\label{eq:taylor}
   \ud(\mathbf{x}) = u(x_1+d_i(x_i),x_2)\approx u(\mathbf{x}) + d_i(x_i)\partial_1 u(\mathbf{x}) ,
\end{equation}
such that
\begin{equation}
\label{eq:disp}
   d_i(x_i) \approx \frac{\ud(\mathbf{x})-u(\mathbf{x})}{\partial_1 u(\mathbf{x})}.
\end{equation}
We aim for simultaneous minimizing the displacement error $d_i$ and maximizing the smoothness of 
$u$ by the minimization of the functional
\begin{equation}
\label{eq:non_convex}
\mathcal{F}(u;\ud) :=  \underbrace{\frac{1}{2}\int_{\R^2}
                           \frac{(u(\mathbf{x})-\ud(\mathbf{x}))^2}{(\partial_1 u(\mathbf{x}))^2}d\mathbf{x}}_{=:\mathcal{D}_2(u,\ud)}
                           +\alpha  
                           \underbrace{\frac{1}{p} \int_{\R^2}(\partial_i^k u(\mathbf{x}))^p d\mathbf{x}}_{=:\mathcal{R}_{i,k,p}(u)},
\end{equation}
with some fixed parameter $\alpha >0$. Our particular choice of the regularization functional is motivated from the 
structure of the data (see Figure \ref{fig:displacement}), where we observed that correcting for line jittering requires 
regularization in $x_2$-direction 
and angular displacements require regularization in $x_1$-direction.

The functional $\mathcal{F}$ is non-convex with respect to $u$ and has a singularity when $\partial_1 u$ vanishes. 
For the practical minimization we consider an approximation consisting of a sequence of convex minimization problems:
\begin{equation}
\left\{
\begin{aligned}
\label{eq:lagged_ite}
u_0 & =  \ud,\\
u_m &:=  \argmin_{u} \mathcal{F}_\ve^c(u;u_{m-1}) \text{ for all } m \in \N\,,
\end{aligned}
\right.
\end{equation}
where $\epsilon>0$ is a small real number, and $\set{\mathcal{F}_\ve^c(\cdot;u_{m-1})}_{m \in \N}$ is the set of convex functionals defined by 
\begin{equation}
\label{eq:convex_eps}
\mathcal{F}^c_\ve(u;u_{m-1}) := 
\frac{1}{2} \int_{\R^2}\frac{(u(\mathbf{x})-u_{m-1}(\mathbf{x}))^2}
                            {(\partial_1 u_{m-1}(\mathbf{x}))^2+\ve}d\mathbf{x} + \alpha \mathcal{R}_{i,k,p}(u).
\end{equation}

In the following we give a convergence result inspired from \cite{ChaMul99}.
\begin{thm}
Let $p=1,2$, $k=1,2$, $\ve >0$,
and let $\set{u_m}_{m \in \N}$ be the sequence of minimizers from \eqref{eq:lagged_ite}. 
\begin{itemize}
\item[1.] The sequences $(\mathcal{F}^c_\ve(u_m,u_{m-1}))_{m \in \N}$ and $(\mathcal{R}_{i,k,p}(u_m))_{m \in \N}$ both are monotonically decreasing.
\item[2.] If there exist a constant $C > 0$ and if $\sup \set{\norm{\partial_1 u_m}_{L^\infty}: m \in \N_0} = C < \infty$, 
          then 
          \[\norm{u_m-u_{m-1}}_{L^2}\rightarrow 0,\;\text{ as }\; m\rightarrow \infty\;.\]
\end{itemize}
\end{thm}

\begin{proof}
From the definition of $\mathcal{F}^c_\ve$ in \eqref{eq:convex_eps} it follows that $\mathcal{F}^c_\ve(u,u_m)$ is proper, 
strictly convex, and lower semi-continuous for every $m \in \N$. Thus there exists an unique minimizer $u_{m+1}$ minimizing $\mathcal{F}^c_\ve(u,u_m)$.

We are able to infer the following inequalities
\[0\leq \mathcal{F}^c_\ve(u_{m+1},u_m)\leq \mathcal{F}^c_\ve(u_m,u_m)=\alpha \mathcal{R}_{i,k,p}(u_m)\leq \mathcal{F}^c_\ve(u_m,u_{m-1}),\]
and
\[\mathcal{R}_{i,k,p}(u_{m+1})\leq \frac{ \mathcal{F}^c_\ve(u_{m+1},u_m)}{\alpha}\leq \frac{\mathcal{F}^c_\ve(u_m,u_m)}{\alpha}= \mathcal{R}_{i,k,p}(u_m),\]
which shows that both $(\mathcal{F}^c_\ve(u_m,u_{m-1}))_{m \in \N}$ and $(\mathcal{R}_{i,k,p}(u_m))_{m \in \N}$ are non-negative and monotonically decreasing.

For the second statement,
\begin{eqnarray*}
\frac{1}{2}\int_{\R^2} \frac{(u_m-u_{m-1})^2 }{\abs{\partial_1 u_{m-1}}^2+\ve}
&= &\mathcal{F}^c_\ve(u_m,u_{m-1}) -\mathcal{F}^c_\ve(u_m,u_m)\\
&\leq& \mathcal{F}^c_\ve(u_m,u_{m-1})-\mathcal{F}^c_\ve(u_{m+1},u_m).
\end{eqnarray*}
From the uniform boundedness assumption of $\set{u_m}$ it follows that
\[\frac{1}{2 (C^2 + \ve)} \norm{u_m-u_{m-1}}^2_{L^2}\leq \mathcal{F}^c_\ve(u_m,u_{m-1})-\mathcal{F}^c_\ve(u_{m+1},u_m).\]
Since the sequence $\set{\mathcal{F}^c_\ve(u_m,u_{m-1})}_{m \in \N}$ is bounded from below and monotonically decreasing
it follows that $\norm{u_m-u_{m-1}}_{L^2}\rightarrow 0$.
\hfill $\Box$
\end{proof}

Identifying $\Delta t = \alpha$, the formal optimality condition for \eqref{eq:convex_eps} 
characterizes the solution of \eqref{eq:lagged_ite} by
\begin{equation} \label{eq:explicit_flow}
 \left\{ \begin{aligned}
  \frac{u_m -u_{m-1}}{\Delta t} &= 
  (\abs{\partial_1 u_{m-1}}^2+\ve) \;(-1)^{k-1} \partial_i^k \left(\frac{\partial_i^k u_m}{\abs{\partial_i^k u_m}^{2-p}} \right),\\
  u_0 &=\ud.
 \end{aligned}\right.
\end{equation}

On the other hand, if $\partial_1 u(\mathbf{x})$ is relatively large,
the estimate of the displacement by the error measure in \eqref{eq:disp} is unrealistic, 
and a least squares error measure $(u(\mathbf{x})-\ud(\mathbf{x}))^2$ might be more efficient. 
To be able to compensate for relatively large and small displacement errors simultaneously, we, therefore, propose to use the 
geometric mean of the two error measures, which is
\begin{equation}
\label{eq:data1}
 \mathcal{D}_1(u,\ud):= \frac{1}{2}\int_{\R^2}\frac{(u(\mathbf{x})-\ud(\mathbf{x}))^2}{\abs{\partial_1 u(\mathbf{x})}}d\mathbf{x}.
\end{equation}
In this case we end up with the following variational model
\begin{equation}
\label{eq:non_convex_II}
\mathcal{F}(u;\ud) :=  \mathcal{D}_1(u,\ud) + \alpha \mathcal{R}_{i,k,p}(u).
\end{equation}
Again identifying $\Delta t = \alpha$, the associated optimality condition for the relaxed functional \eqref{eq:non_convex_II} is
\begin{equation} \label{eq:explicit_flow_II}
 \left\{ \begin{aligned}
  \frac{u_m -u_{m-1}}{\Delta t} &= 
  (\abs{\partial_1 u_{m-1}}+\ve) \;(-1)^{k-1} \partial_i^k \left(\frac{\partial_i^k u_m}{\abs{\partial_i^k u_m}^{2-p}} \right),\\
  u_0 &=\ud.
 \end{aligned}\right.
\end{equation}

\section{Nonlinear flows}
\label{sec:flows}
Let $i=1,2$, $k=1,2$, $p=1,2$, $q=1,2$ be fixed and assume that $\ud$ is given.
$u_m$ solving \eqref{eq:explicit_flow} (corresponding to $q=2$), \eqref{eq:explicit_flow_II} (corresponding to $q=1$) can be 
considered a numerical approximation of the flow
\begin{equation} \label{eq:formal_flow}
\left\{ \begin{aligned}
 \dot{u} &= (-1)^{k-1}\abs{\partial_1 u}^q \partial_i^k \left(\frac{\partial_i^k u}{\abs{\partial_i^k u}^{2-p}} \right)\quad 
 \text{ in } \R^2\times (0,\infty),\\
 u &=\ud \text{ in } \R^2  \times \set{0}
\end{aligned}\right.
\end{equation}
at time $t = m\Delta t$. Here $u=u(\mathbf{x},t)$ and $\dot{u}$ denotes the derivative of $u$ with respect to $t$. 
We also can consider \eqref{eq:formal_flow} as the flow according to the non-convex functional $\mathcal{F}$ (which is depended on $i$).

\begin{rem}
In practical simulations the unbounded domain $\R^2$ is replaced by $\Omega = (0,1)^2$ and the flow 
is associated with boundary conditions:
\begin{equation*}
 \begin{aligned}
  \partial_i^{2l-1} u& = 0,\text{ on }\set{0,1} \times (0,1)  , \text{ for all } l=1,..,k \text{ and for } i=1;\\  \partial_i^{2l-1} u & = 0,\text{ on } (0,1) \times \set{0,1}, \text{ for all } l=1,..,k \text{ and for } i=2.
 \end{aligned}
\end{equation*}

The case of $i=2$ has been considered in \cite{DonPatSchOek15}.
 
When $i=1$, the right hand side of \eqref{eq:formal_flow} only involves partial 
derivatives of $u$ with respect to $x_1$, such that it reduces to a system of independent 
equations defined for the functions $u(\cdot,x_2)$ for every $x_2 \in (0,1)$.
\end{rem}

\subsection{Properties of the flows}
In the following, we present some numerical simulations with \eqref{eq:formal_flow}.
\begin{itemize}
 \item For $i=1$ \eqref{eq:formal_flow} reads as follows
\begin{equation} \label{eq:flow1}
 \dot{u} = (-1)^{k-1}\abs{\partial_1 u}^q \partial_1^k \left(\frac{\partial_1^k u}{\abs{\partial_1^k u}^{2-p}} \right)\;.
\end{equation}

Figure \ref{fig:diffusion} shows numerical simulations of \eqref{eq:flow1} for different choices of $k$ and $p$ with $q=2$. 
In all test cases the initial data $\ud$ is a function representing a narrow white strip on a black background.
We visualize the solutions at $t=10^{-6}$ with identical time unit. 
We observe diffusion of the white strip for all choices 
of $k$ and $p$ except in the case $k=1, p=1$, with $q=2$.

We emphasis that when $k=1,p=2$ and $q=2$ \eqref{eq:flow1} reads as follows
\begin{equation}
\label{eq:flow1_2}
\dot{u} = \abs{\partial_1 u}^2 \partial_1^2 u =\frac{\partial_1 ((\partial_1 u)^3) }{3},
\end{equation}
of which the differential operator on the right-hand side is a one dimensional \emph{p-Laplacian} 
with $p=4$ (see \cite{Uhl77} for some regularity properties, where they have a general result for \emph{p-Laplacian} in $\R^n$). 
Hence the equation \eqref{eq:flow1_2} is nothing but a system of independent one dimensional \emph{p-Laplacian} flow,
of which the equilibrium (stationary point) is a function of the form $\tilde{u}=c_1(x_2) x_1+ c_2(x_2)$.
\begin{figure}
\begin{center}
\includegraphics[width=0.3\textwidth]{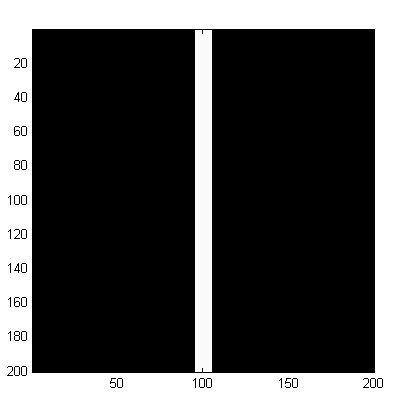}\qquad \includegraphics[width=0.3\textwidth]{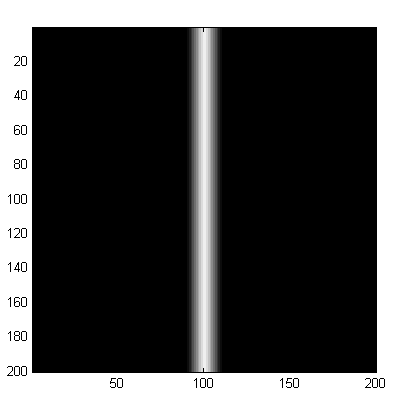}  \\
\includegraphics[width=0.3\textwidth]{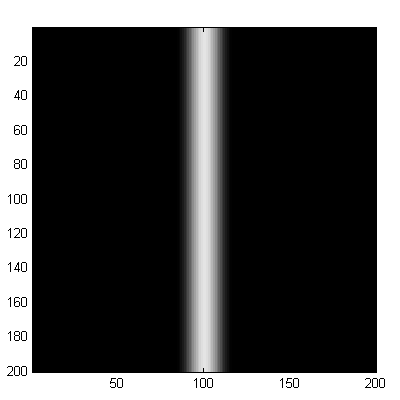} \includegraphics[width=0.3\textwidth]{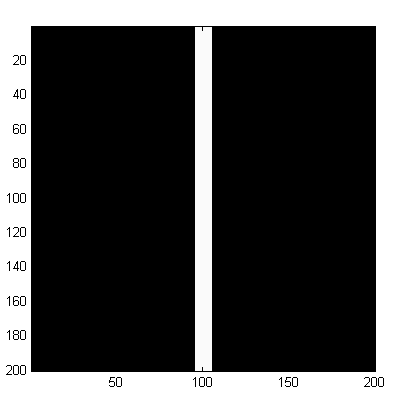} \includegraphics[width=0.3\textwidth]{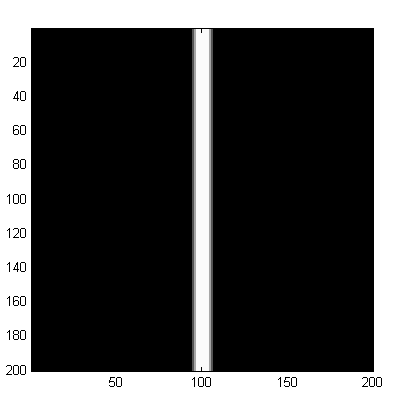} 
\end{center}
\caption{The solutions of the PDEs \eqref{eq:flow1} with $q=2$ at time $t=10^{-6}$. From top to bottom, and from left to right, the images are corresponding to the initial value $\ud$ and the results of the evolution with various parameters: $k=1, p=2$; $k=2, p=2$; $k=1, p=1$; $k=2, p=1$ in \eqref{eq:flow1}. We consider the $x_1$ lines are discretized with mesh size $\Delta x_1=0.1$.}
\label{fig:diffusion}
\begin{center}
\includegraphics[width=0.32\textwidth]{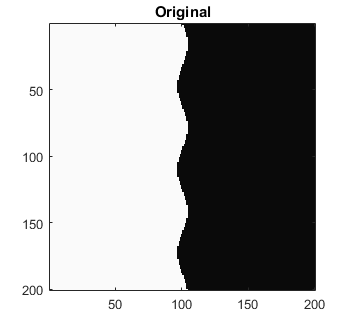}\qquad \includegraphics[width=0.32\textwidth]{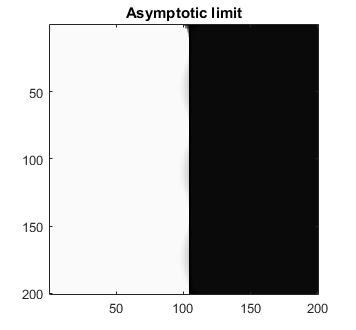}  \\
\includegraphics[width=0.32\textwidth]{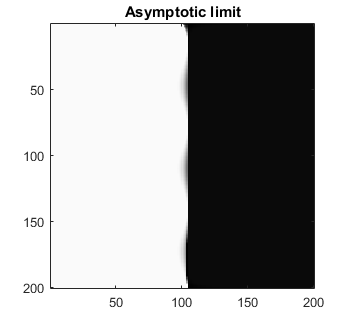} \includegraphics[width=0.32\textwidth]{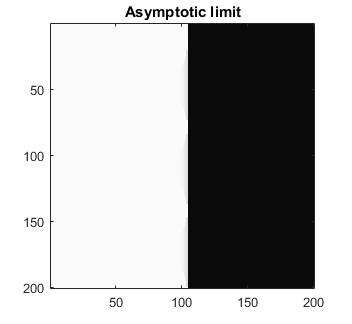} \includegraphics[width=0.32\textwidth]{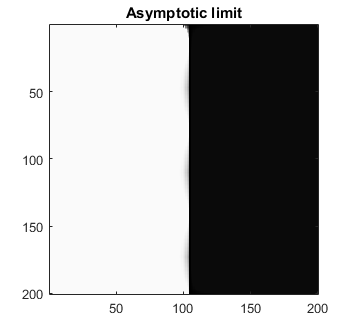}
\end{center}
\caption{The long time behaviors of the PDEs in \eqref{eq:flow2} with $q=2$. From top to bottom, and from left to right, the images are corresponding to the initial value $\ud$ and the results of the evolution with $q=2$ and various parameters: 
$k=1, p=2$; $k=2, p=2$; $k=1, p=1$; $k=2, p=1$ in \eqref{eq:flow2}. }
\label{fig:lineflow}
\end{figure}
\item
For $i=2$ \eqref{eq:formal_flow} reads as follows
\begin{equation} \label{eq:flow2}
 \dot{u} = (-1)^{k-1}\abs{\partial_1 u}^q \partial_2^k \left(\frac{\partial_2^k u}{\abs{\partial_2^k u}^{2-p}} \right).
\end{equation} 
We investigate the long time behavior of the solution of equation \eqref{eq:flow2} initialized with some curved 
interface data $\ud$. The numerical results are presented in Figure \ref{fig:lineflow}.
We find that the curved interface evolves into a vertical line (for all cases of $k$ and $p$, with $q=2$, we have tested), 
which we assume to be a general analytical property of the PDEs.
\end{itemize}

\subsection{Angular correction in tomography}
Now, we are applying the displacement correction methods for problems arising in tomography.
As data we consider angular disturbed sinograms of the Radon transform of 2D images and the final goal 
is to reconstruct the corresponding attenuation coefficients.
If the sinogram is recorded with angular perturbations, without further processing, application of the inverse Radon 
transform may produce outliers in the reconstruction (see Figure \ref{fig:phanton}). 
\begin{figure}
\includegraphics[width=0.5\textwidth]{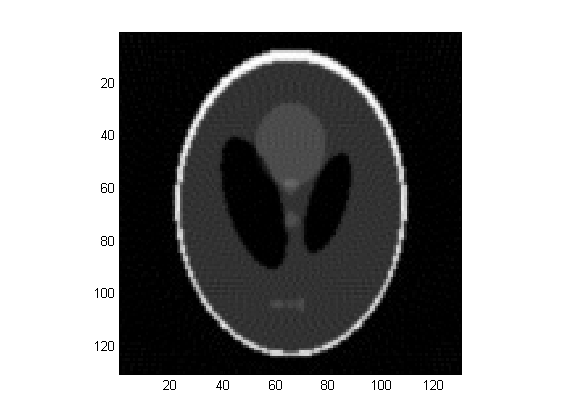} \includegraphics[width=0.5\textwidth]{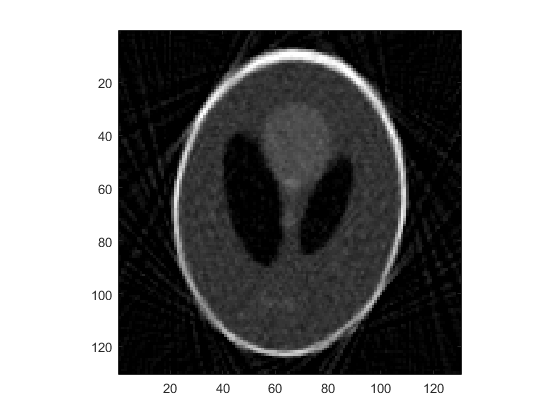}
\caption{Left: Phantom image. Right: A direct reconstruction with a filter backprojection (FBP) algorithm 
from an angular perturbed sinogram 
(see Figure \ref{fig:displacement}).}
\label{fig:phanton}
\end{figure}

As a test image we use the Shepp-Logan phantom of size $128 \times 128$, and discretize the angular axis of the sinogram 
domain $[0,\pi)$ with uniform step size $\frac{1}{90}\pi$. 
The synthetic data are generated by evaluating the Radon transform along the line in direction $\hat{\theta}=\theta+ d_1(\theta)$.
The (inverse) Radon transforms are implemented with the Matlab toolbox.

In the first series of experiments (see Figure \ref{fig:s_angular_correction}), 
we allow the random perturbations on the beam directions to be $d_1(\theta)\in [0,\frac{1}{30}\pi]$, 
which is relatively small.
In the example shown in Figure \ref{fig:s_angular_correction} these sampling errors do not cause 
significant mismatch in the reconstruction.
\begin{figure}
\includegraphics[width=0.5\textwidth]{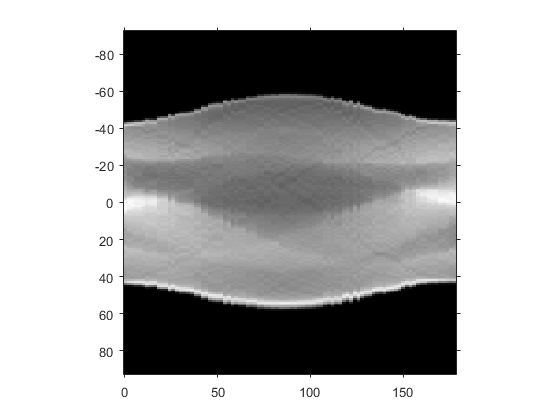} \includegraphics[width=0.5\textwidth]{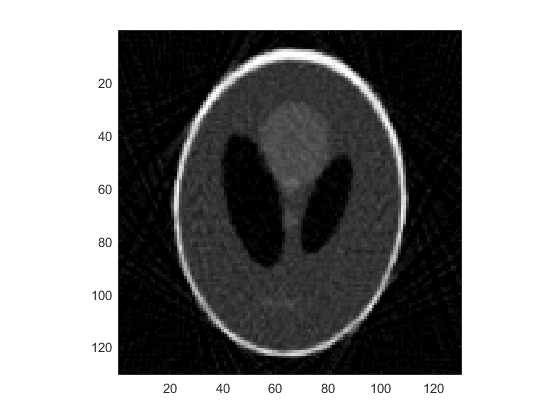}\\
\includegraphics[width=0.5\textwidth]{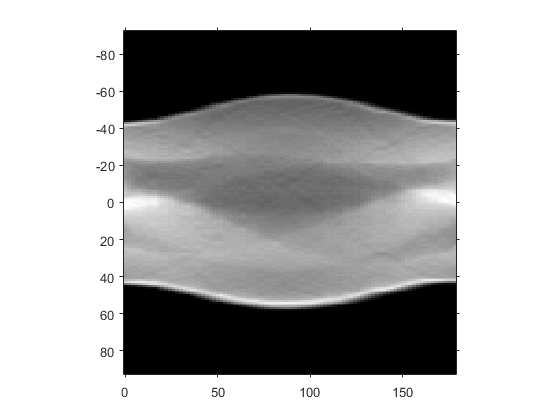}  \includegraphics[width=0.5\textwidth]{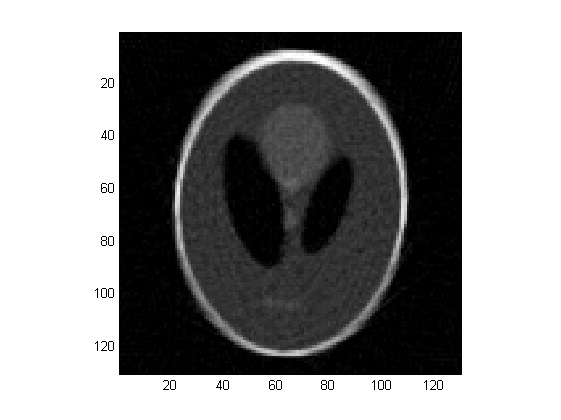} 
\caption{Top Left: Sinogram with relatively small angular perturbations ($d_1(\theta)\in [0,\frac{1}{30}\pi]$). Top Right: Direct reconstruction with a FBP algorithm.
Bottom: The applications of the flow \eqref{eq:flow1} for angular correction with parameter settings $k=1, p=2$ and $q=2$. 
}
\label{fig:s_angular_correction}
\includegraphics[width=0.5\textwidth]{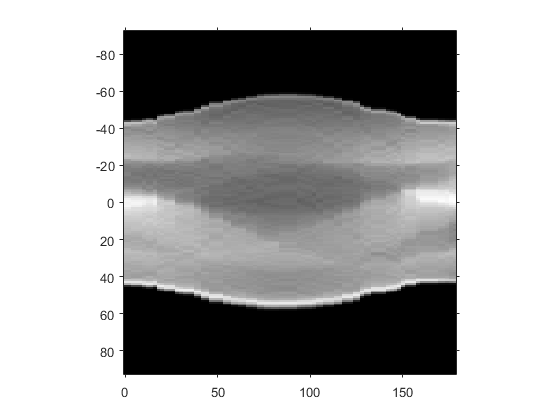} \includegraphics[width=0.5\textwidth]{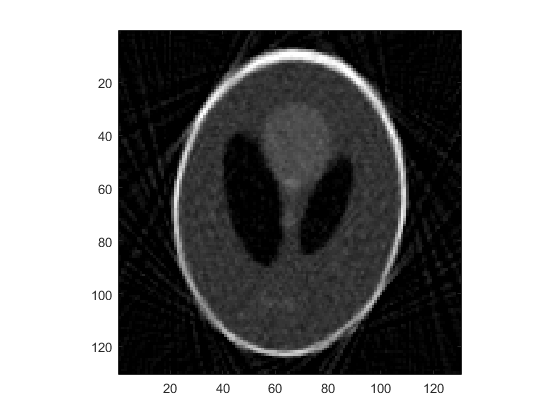}\\
\includegraphics[width=0.5\textwidth]{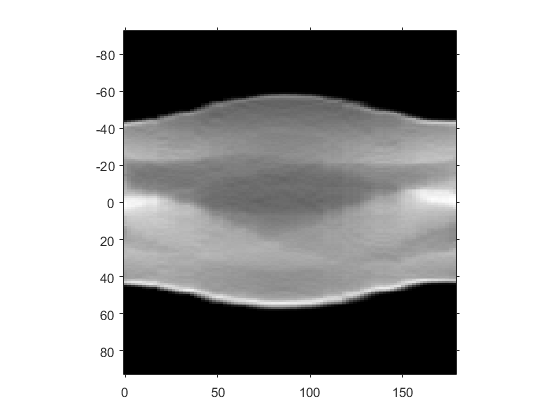}  \includegraphics[width=0.5\textwidth]{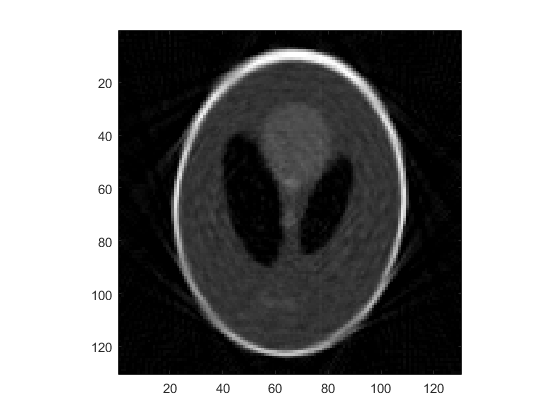} 
\caption{ Top: The results obtained with an heuristic discrete optimization algorithm method for solving \eqref{eq:problem1}.  
Bottom: The applications of the flow \eqref{eq:flow1a} for angular correction.}
\label{fig:l_angular_correction}
\end{figure}

Although the flows in \eqref{eq:flow1} are suitable for correcting small displacement errors,
they may not be very qualified for the data with larger displacement errors and with additive noise.
In the latter case, the recorded sinogram is considered to be
\begin{equation}\label{eq:angle_dis_noise}
y^{\delta}(\theta,l)=R[f](\theta+d_1(\theta),l)+\eta,
\end{equation}
where $\eta$ denotes some additive noise.
Here the filtering by the flow defined in \eqref{eq:formal_flow} with $q=1$ 
outperforms the filtering by the flow with $q=2$.
In the numerical experiments, we tested with an example of the angular perturbation $d_1(\theta)\in [0,\frac{1}{18}\pi]$ 
(see Figure \ref{fig:displacement}). For $q=1$ we show the filtering by \eqref{eq:formal_flow} with $k=1$ and $p=2$.
That is, we use the equation
\begin{equation} \label{eq:flow1a}
 \dot{u} = \abs{\partial_1 u} \partial_1^2 u,
\end{equation}
for filtering. 
Our numerical results are reported in Figure \ref{fig:l_angular_correction} and Figure \ref{fig:l_noisy_angular_correction},
which is based on the displacement error shown in Figure \ref{fig:displacement}. 
As a comparison, we also show the results when Problem \eqref{eq:problem1} is considered a discrete optimization problem and 
is solved with a heuristic discrete optimization algorithm, which is generalized from \cite{Nik09}. The right one on the top of Figure \ref{fig:l_angular_correction} 
shows the final result, having no preference against standard FBP.
The bottom images in Figure \ref{fig:l_angular_correction} show the results which are obtained by applying the nonlinear flow \eqref{eq:flow1a} 
for correcting the unknown angular perturbations.
In this example, the nonlinear flows have a better performance in comparing with the heuristic discrete optimization algorithm.
From the results shown in Figure \ref{fig:l_noisy_angular_correction}, 
we find that the proposed equation \eqref{eq:flow1a} is able to correcting the angle displacement error and denoise simultaneously. 
\begin{figure}
\includegraphics[width=0.5\textwidth]{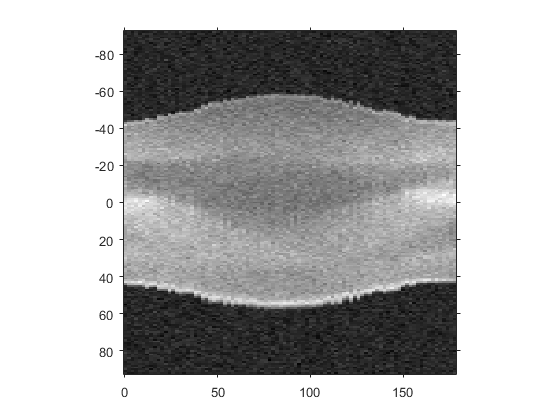}   \includegraphics[width=0.5\textwidth]{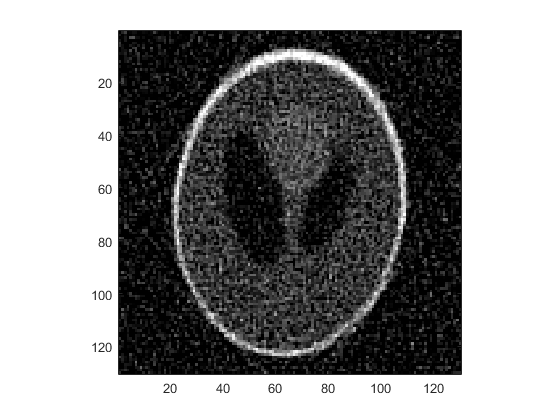} \\
\includegraphics[width=0.5\textwidth]{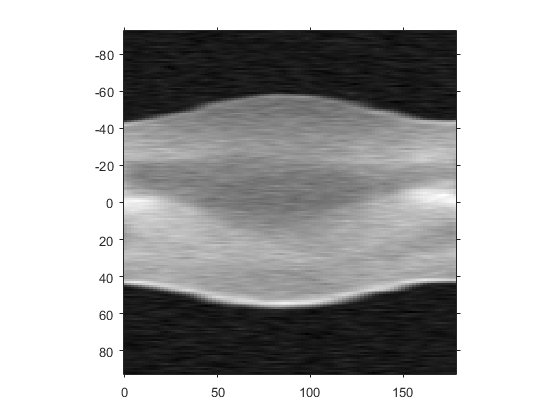} \includegraphics[width=0.5\textwidth]{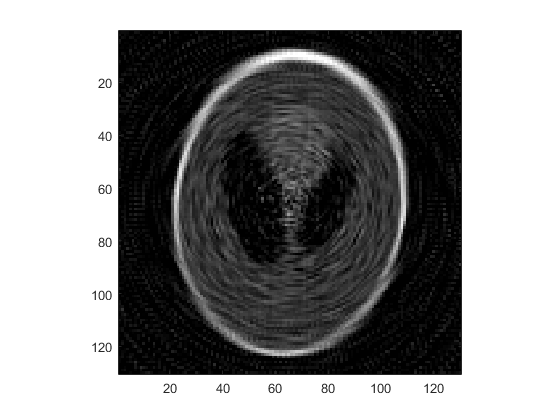} 
\caption{The applications of the flow \eqref{eq:flow1a} for angular correction from data containing both displacement and additive errors.}
\label{fig:l_noisy_angular_correction}
\end{figure}

\section{Conclusion}
In this paper, we have considered two families of nonlinear flows for correcting two different kinds of 
displacement errors. Our numerical analysis revealed interesting properties of the flows,
which should be confirmed by further theoretical studies in the future.
We have also presented some applications to tomography, where the novel PDEs are able to correct for 
angular displacement errors.
Some other application area of our methods are electron tomography in single particle analysis \cite{NatWue01,Oek15}.

\section*{Acknowledgements}
The work of OS has been supported by the Austrian Science Fund (FWF): Geometry and Simulation, project S11704 (Variational methods for imaging on manifolds), and Interdisciplinary Coupled Physics Imaging, project P26687-N25.

\end{document}